\documentclass[a4paper,11pt]{amsart}
 \usepackage{float}
\usepackage{amssymb}
\usepackage{mathrsfs}
\usepackage{amsmath,amssymb,amsthm,latexsym,amscd,mathrsfs}
\usepackage{indentfirst}
\usepackage{stmaryrd}
\usepackage{graphicx}
\usepackage{extarrows}
\usepackage{multirow}
\usepackage{latexsym}
\usepackage{amsfonts}
\usepackage{color}
\usepackage{pictexwd,dcpic}
\usepackage{graphicx}
\usepackage{psfrag}
\usepackage{hyperref}
\usepackage{comment}
\usepackage{enumerate}

\numberwithin{equation}{section} \theoremstyle{plain}
\newtheorem{thm}{Theorem}[section]

\newtheorem{prop}[thm]{Proposition}
\newtheorem{lem}[thm]{Lemma}
\newtheorem{cor}[thm]{Corollary}
\newtheorem{defn}[thm]{Definition}

\newtheorem{rem}[thm]{Remark}

\newtheorem*{acknow}{Acknowledgments}
   
 \makeatletter

\def\<{\langle}
\def\>{\rangle}
\def\({\left(}
\def\){\right)}
\def\[{\left[}
\def\]{\right]}
\def\tr{\mathop{\text{tr}}}

\makeatother

\title{An isoperimetric inequality of minimal hypersurfaces in spheres}

\author[F.G. Li]{Fagui Li${}^{*1,2}$}
\address{$^{1}$
China Beijing International Center for Mathematical Research, Peking University, Beijing 100871, P.R. CHINA.}
\email{faguili@bicmr.pku.edu.cn}
\address{$^{2}$
School of Mathematical Sciences, Laboratory of Mathematics and Complex Systems, Beijing Normal University, Beijing 100875, P.R. CHINA.}
\email{faguili@mail.bnu.edu.cn}

\author[N. Chen]{Niang Chen${}^{+}$}
\address{$^{+}$
Department of Mathematics, Faculty of Arts and Sciences, Beijing Normal University, Zhuhai 519087, China}
\email{chenniang@bnu.edu.cn}

\date{}
\thanks {$^{*}$ the corresponding author.}
\thanks{F. G. Li is partially supported by  NSFC (No. 12171037, 12271040) and China Postdoctoral Science Foundation (No. 2022M720261).}
\thanks{N. Chen  is partially supported by the start-up funding for research provided
by Beijing Normal University (No. 310432120).}

\begin{document}
\maketitle

\begin{abstract}
Let $ M^n$ be a closed immersed minimal  hypersurface   in  the unit sphere $\mathbb{S}^{n+1}$.
We establish a  special isoperimetric inequality of $M^n$. As an application, if the scalar curvature of $ M^n$ is constant,  then we  get a uniform lower bound independent of $M^n$ for the  isoperimetric inequality. In addition,  we obtain an inequality between Cheeger's isoperimetric constant  and the volume of the nodal set of the height function.
\end{abstract}

\section{Introduction}
The isoperimetric inequalities have always been an important subject in differential geometry and they are  bridges of analysis and geometry. There are some elegant works on isoperimetric inequalities (cf. \cite{Brendle S The isoperimetric inequality JMS,J Choe and Gulliver 1992,Osserman Robert 1980 isoperimetric,Yau Isoperimetric constants 1975}, etc).

Let $x: M^n\looparrowright   \mathbb{S}^{n+1} \subset \mathbb{R}^{n+2}$ be a closed immersed minimal hypersurface  in the unit sphere and denote by $\nu(x)$ a (local) unit normal vector field of $M^n$,
$\nabla$ and $\overline\nabla$  be the Levi-Civita connections on $M^n$ and  $\mathbb{S}^{n+1}$,  respectively. Let $A$ be the shape operator with respect to $\nu$, i.e., $A(X)=-\overline\nabla_X\nu$. The squared length of the second fundamental form is $S=\|A\|^2$.
For any unit vector $a \in \mathbb{S}^{n+1}$,  the height functions are defined as
\begin{equation*}\label{height functions}
\varphi_a(x) = \langle x,a \rangle,\quad  \psi_a (x)=\langle \nu,a \rangle.
\end{equation*}
These two functions are very basic and important. For instance,
 the well known Takahashi theorem \cite{Takahashi 1966} states that \emph{$M^n$ is minimal if and only if there exists a constant $\lambda$ such that $\Delta \varphi_a =-\lambda\varphi_a$
for all $a \in \mathbb{S}^{n+1}$}. Analogously, Ge and Li \cite{Ge Li 2020} gave a Takahashi-type theorem, i.e., \emph{an immersed hypersurface $M^n$ in  $\mathbb{S}^{n+1}$  is minimal and has constant scalar curvature (CSC) if and only if $\Delta \psi_a=\lambda\psi_a$ for some constant $\lambda$ independent of $a \in \mathbb{S}^{n+1}$.} 
This condition is linked to the famous Chern Conjecture (cf. \cite{Chang 1993,M Scherfner S Weiss and  Yau 2012,TWY18,TY20,Cheng Qing Ming and H Yang 1998}, etc), which states that \emph{a closed  immersed minimal CSC hypersurface of $\mathbb{S}^{n+1}$  is isoparametric.}

Let  ${\left\lbrace |\varphi_a|\geq t\right\rbrace }={\left \lbrace x\in M^n: |\varphi_a|\geq t\right\rbrace }$ and  ${\left\lbrace |\varphi_a|= t\right\rbrace }={\left \lbrace x\in M^n: |\varphi_a|= t\right\rbrace }$. In particular,
due to  $\Delta\varphi_a=-n\varphi_a$ and  $ a \in \mathbb{S}^{n+1}$,
 $${\left\lbrace \varphi_a= 0\right\rbrace }={\left \lbrace x\in M^n: \varphi_a= 0\right\rbrace }$$
is the nodal set of  the eigenfunction $\varphi_a$. Here,
the zero set of the eigenfunction of an elliptic operator, and its complement are called the nodal set, and nodal domain, respectively.
Suppose
  $
  S_{\max}=\sup_{p\in M^n}S(p),
  $
$$
\theta_1=\frac{\int_{M}S}
{2nS_{\max} {\rm Vol }\left( M^n \right) },\ \
\theta_2=\frac{n}{4n^2-3n+1}
 \frac{
 \left( {\int_{M} }S\right) ^2}
{ {\rm Vol }\left( M^n \right) \int_{M} S^2},
$$
and
$$
C_1=\max\{\theta_1,\theta_2\},\ \
C_2=\inf_{s\leq r\leq 1}\frac{2+nr\ln\left(\frac{1-s^2}{1-r^2} \right)}{2+n\ln\left(\frac{1-s^2}{1-r^2} \right)  }.
$$
We use ${\rm Vol}$ to represent the volume measure in this paper and
  the following special isoperimetric inequality is the main result.
\begin{thm}\label{thm Volume estimation of minimal hypersurface varphi2a and varphia}
Let $ M^n$ be a closed immersed, non-totally geodesic,  minimal  hypersurface   in  $\mathbb{S}^{n+1}$.
\begin{itemize}
\item [(i)]
For all $0 \leq s<1$ and $a\in \mathbb{S}^{n+1}$, the following inequality holds:
$$
 {\rm Vol }{\left\lbrace |\varphi_a|=s\right\rbrace }\geq
 C(n,s,S)
{\rm Vol }{\left\lbrace |\varphi_a|\geq s\right\rbrace },
$$
where
$$
 C(n,s,S)=
 \left\{
\begin{array}{rcl}
    \frac{nC_1}{2C_2},& \ \  &s=0;\\
    \frac{nC_1}{C_2\sqrt{1-s^2}},&\ \   &0<s\leq \min\left\lbrace \sqrt{C_1},\frac{C_1}{C_2}\right\rbrace;\\
    \frac{ns}{\sqrt{1-s^2}},&\ \   &\min\left\lbrace \sqrt{C_1},\frac{C_1}{C_2}\right\rbrace< s<1.
\end{array}
 \right.
$$
\item [(ii)]
$$
\frac{\left( n+1\right) {\rm Vol}\left( \mathbb{S}^{n+1}\right)
}{n{\rm Vol}\left(\mathbb{S}^{n} \right)}
\sup_{a\in \mathbb{S}^{n+1}}
 {\rm Vol }{\left\lbrace \varphi_a=0\right\rbrace }
 \geq{\rm Vol}(M^n).
$$
  \end{itemize}
\end{thm}

Obviously, if $ M^n$ is a closed immersed minimal CSC hypersurface $($non-totally geodesic$)$   in  $\mathbb{S}^{n+1}$, then $C_1=\theta_1=\frac{1}{2n}$ in Theorem  \ref{thm Volume estimation of minimal hypersurface varphi2a and varphia} 
 and one has
\begin{cor}\label{cor Volume estimation of minimal hypersurface varphi2a and varphia S equiv constant}
Let $ M^n$ be a closed immersed, non-totally geodesic,  minimal CSC  hypersurface   in  $\mathbb{S}^{n+1}$. Then for all $0 \leq s<1$ and $a\in \mathbb{S}^{n+1}$, the following inequality holds:
$$
 {\rm Vol }{\left\lbrace |\varphi_a|=s\right\rbrace }\geq
 C(n,s)
{\rm Vol }{\left\lbrace |\varphi_a|\geq s\right\rbrace },
$$
where
$$
 C(n,s)=
 \left\{
\begin{array}{rcl}
    \frac{1}{4C_2},& \ \  &s=0;\\
    \frac{1}{2C_2\sqrt{1-s^2}},&\ \   &0<s\leq \min\left\lbrace \sqrt{\frac{1}{2n}},\frac{1}{2nC_2}\right\rbrace;\\
    \frac{ns}{\sqrt{1-s^2}},&\ \   &\min\left\lbrace \sqrt{\frac{1}{2n}},\frac{1}{2nC_2}\right\rbrace< s<1.
\end{array}
 \right.
$$
\end{cor}

More precisely, Corollary \ref{cor Volume estimation of minimal hypersurface varphi2a and varphia S equiv constant}   implies that the condition of constant scalar curvature has strong rigidity for minimal  hypersurfaces, since the constant $C(n,s)$ depends only on $n$ and $s$.
Hence, the volume of $M^n$ is strongly restricted by the volume of    nodal set of  the eigenfunctions $\varphi_a$  ($a\in \mathbb{S}^{n+1}$) for minimal CSC hypersurfaces (non-totally geodesic), i.e.,
$$
C_0(n)
 {\rm Vol }{\left\lbrace \varphi_a=0\right\rbrace }\geq
{\rm Vol }
\left(  M^n\right),
$$
where
$
C_0(n)=C(n,0)=4\inf_{0\leq r\leq 1}\frac{2-nr\ln\left(1-r^2 \right)}{2-n\ln\left(1-r^2 \right)  }.
$
 Besides, this rigid property  provides some evidence for the Chern Conjecture.

\begin{rem}\label{Integral-Einstein isoper inequalities}
Under the conditions of Corollary \ref{cor Volume estimation of minimal hypersurface varphi2a and varphia S equiv constant}, if $M^n$ is an Integral-Einstein  (see Definition \ref{defintion integral Eins manifold}) minimal  CSC  hypersurface in $\mathbb{S}^{n+1}$ (or  CSC hypersurface  with $S>n$ and constant third mean curvature), then the constant $C(n,s)$ can be improved (see Corollary \ref{cor Volume estimation of Einstein minimal hypersurface varphi2a and varphia S equiv constant}).
\end{rem}
 In 1984, Cheng-Li-Yau \cite{Cheng Li Yau 1984 Heat equations} proved  that \emph{if $M^n$ is a closed immersed minimal hypersurface in  $\mathbb{S}^{n+1}$ and $M^n$ is non-totally geodesic, then}
 \begin{equation*}
    {\rm Vol}(M^n)>\left( 1+\frac{3}{\widetilde{B}_n}\right) {\rm Vol}(\mathbb{S}^n),
 \end{equation*}
where $\widetilde B_n=2n+3+2\exp \left( 2n\widetilde C_n\right)$ and
  $\widetilde C_n= \frac{1}{2}n^{n/2}e\Gamma(n/2,1)$. 
Thus, we have
 \begin{cor}\label{thm Volume estimation of minimal hypersurface varphi0}
 Let $ M^n$ be a closed immersed, non-totally geodesic,  minimal CSC  hypersurface   in  $\mathbb{S}^{n+1}$. Then there is  a positive constant
   $\epsilon(n)>0$, depending only on $n$,
   such that
   $$
    {\rm Vol }{\left\lbrace \varphi_a=0\right\rbrace }\geq
   \epsilon(n){\rm Vol}(\mathbb{S}^n)\quad \textit{for all }  a\in \mathbb{S}^{n+1},
   $$
where $\epsilon(n)>\frac{1}{4}\left( 1+\frac{3}{\widetilde B_n}\right)
   \sup_{0\leq r\leq 1}
   \frac{2-n\ln\left(1-r^2 \right)  }{2-nr\ln\left(1-r^2 \right)}
   $.
 \end{cor}
 Let $h(M)$ denote the Cheeger isoperimetric constant (see Definition \ref{Definition Cheeger isoperimetric constant}), we have 
 \begin{thm}\label{Cheeger isoperimetric constant}
 Let $M^n$ be a closed immersed, non-totally geodesic, minimal  hypersurface   in  $\mathbb{S}^{n+1}$. 
 Then for all  $a\in \mathbb{S}^{n+1}$ we have
 \begin{equation*}
  {\rm Vol }{\left\lbrace \varphi_a=0\right\rbrace }\geq
  \frac{2\sqrt{n+1}C_1}{C_0(n)}h(M) {\rm Vol }{\left( M^n\right) }.
 \end{equation*}
In particular, we have the following assertions:
\begin{itemize}
\item [(i)] If  $M^n$ is embedded, then  
$h(M)>\frac{-\delta(n-1)+\sqrt{\delta^2(n-1)^2+5n}}{10},$
 where  $\delta=\sqrt{\frac{S_{\max}-n}{n}};$
\item [(ii)] If the image of  $M^n$ is invariant under the antipodal map (i.e., $M^n$ is radially symmetrical), then
 $
  {\rm Vol }{\left\lbrace \varphi_a=0\right\rbrace }
  \geq \frac{1}{2}h(M){{\rm Vol }
    \left(  M^n\right)}.
 $
 \end{itemize}
 \end{thm}
\section{Preliminary lemmas}
In this section, we will prove Lemma \ref{lem Volume estimation of minimal hypersurface varphi2a and varphia} by Proposition  \ref{prop funda} and Lemma \ref{lem Volume estimation of minimal hypersurface}.
A direct calculation shows:
\begin{prop} \cite{Ge Li 2020, Nomizu and Smyth 1969} \label{prop funda}
For all $a\in \mathbb{S}^{n+1}$, we have
$$\begin{array}{lll}
\nabla \varphi_a=a^{\rm T},&
\nabla \psi_a =-A(a^{\rm T}),\\
\Delta \varphi_a=-n\varphi_a+nH\psi_a ,&
\Delta \psi_a =-n\left\langle \nabla H, a \right\rangle +nH\varphi_a -S\psi_a. \\
\end{array}$$
 where  $a^{\rm T}\in \Gamma(TM)$ denotes the tangent component of $a$ along $M^n$; $A$ is the shape operator with respect to $\nu$, i.e., $A(X)=-\overline\nabla_X\nu$; $S=\|A\|^2={\tr} \left( AA^t\right) $ and $H = \frac{1}{n} \tr A$ is the mean curvature.
\end{prop}

\begin{lem} \cite{Ge Li 2020} \label{lem Volume estimation of minimal hypersurface}
Let $ M^n$ be a closed immersed  minimal hypersurface  in  $\mathbb{S}^{n+1}$ with the squared length of the second fundamental form $S$.
\begin{itemize}
\item[(i)]
If $ S\not\equiv0$, then
$$
\frac{\int_{M}S}
{2n S_{\max}}
\leq
\inf_{a\in \mathbb{S}^{n+1}}\int_{M} \varphi^2_a.
$$
The equality holds if and only if $ S\equiv n$ and $M^n$  is the minimal Clifford torus $S^{1}(\sqrt{\frac{1}{n}})\times S^{n-1}(\sqrt{\frac{n-1}{n}})$.
\item[(ii)] If $ S$ has no restrictions, then
$$
\frac{n}{4n^2-3n+1} 
\left( {\int_{M} }S\right) ^2
\leq \int_{M} S^2
\inf_{a\in \mathbb{S}^{n+1}}\int_{M} \varphi^2_a.
$$
The equality holds if and only if $M^n$ is an equator.
\end{itemize}
\end{lem}

\begin{lem}\label{lem Volume estimation of minimal hypersurface varphi2a and varphia}
Let $ M^n$ be a closed immersed, non-totally geodesic,  minimal hypersurface   in  $\mathbb{S}^{n+1}$. Then for all $0 \leq s\leq r\leq 1$ and $a\in \mathbb{S}^{n+1}$, the following inequality holds:
$$
\int_{\left\lbrace |\varphi_a|\geq s\right\rbrace}\varphi^2_a\leq
\frac{2+nr\ln\left(\frac{1-s^2}{1-r^2} \right)}{2+n\ln\left(\frac{1-s^2}{1-r^2} \right)  }
\int_{\left\lbrace |\varphi_a|\geq s\right\rbrace}|\varphi_a|.
$$
\end{lem}
\begin{proof}
By Proposition \ref{prop funda}, we have
$$\nabla \varphi_a=a^{\rm T},\ \ \Delta \varphi_a=-n\varphi_a,$$
for all $a\in \mathbb{S}^{n+1}$.
Hence, by the divergence theorem and
\begin{equation}\label{equation a2 psi2 vaephi2}
|a^{\rm T}|^2+\varphi_a^2+\psi_a^2=1,
\end{equation}
for all $ 0<t\leq 1$ one has
\begin{equation}\label{equation Stokess formula to varphia}
\int_{\left\lbrace |\varphi_a|\geq t\right\rbrace }
|\varphi_a|=
\int_{\left\lbrace |\varphi_a|= t\right\rbrace}\frac{|a^{\rm T}|}{n}=
\int_{\left\lbrace |\varphi_a|= t\right\rbrace}\frac{\sqrt{1-\varphi_a^2-\psi_a^2}}{n}\leq
\int_{\left\lbrace |\varphi_a|= t\right\rbrace}\frac{\sqrt{1-t^2}}{n},
\end{equation}
where ${\left\lbrace |\varphi_a|\geq t\right\rbrace }={\left \lbrace x\in M^n: |\varphi_a|\geq t\right\rbrace }$ and  ${\left\lbrace |\varphi_a|= t\right\rbrace }={\left \lbrace x\in M^n: |\varphi_a|= t\right\rbrace }$.
Due to the co-area formula and (\ref{equation a2 psi2 vaephi2}, \ref{equation Stokess formula to varphia}), for all $0 \leq s< r\leq 1$ we obtain
\begin{equation}\label{equation leqr to geqr varphia}
\begin{aligned}
\int_{\left\lbrace s\leq|\varphi_a| \leq r\right\rbrace }
|\varphi_a|
&=
\int_s^r
\int_{\left\lbrace |\varphi_a|=t\right\rbrace }
\frac{|\varphi_a|}{|a^{\rm T}|}
=
\int_s^r
\int_{\left\lbrace |\varphi_a|=t\right\rbrace }
\frac{|\varphi_a|}{\sqrt{1-\varphi_a^2-\psi_a^2}}\\
&\geq
\int_s^r
\int_{\left\lbrace |\varphi_a|=t\right\rbrace }
\frac{t}{\sqrt{1-t^2}}
\geq
\int_s^r
\int_{\left\lbrace |\varphi_a|\geq t\right\rbrace }
\frac{t}{\sqrt{1-t^2}}
\frac{n}{\sqrt{1-t^2}}
|\varphi_a|\\
&=
\int_s^r
\int_{\left\lbrace |\varphi_a|\geq t\right\rbrace }
\frac{nt}{{1-t^2}}
|\varphi_a|
\geq
\int_{\left\lbrace |\varphi_a|\geq r\right\rbrace }
|\varphi_a|
\int_s^r\frac{nt}{{1-t^2}}\\
&=\frac{n}{2}\ln\left(\frac{1-s^2}{1-r^2} \right)
\int_{\left\lbrace |\varphi_a|\geq r\right\rbrace }
|\varphi_a|.
\end{aligned}
\end{equation}
For all $0 \leq s< r\leq 1$,  by $0\leq\varphi_a^2\leq|\varphi_a|\leq1$ we have
\begin{equation}\label{equation leqr to  varphia2 1-r}
\begin{aligned}
\int_{\left\lbrace |\varphi_a|\geq s\right\rbrace }\varphi_a^2
&=
\int_{\left\lbrace |\varphi_a|\geq r\right\rbrace }
\varphi_a^2+
\int_{\left\lbrace s\leq |\varphi_a|< r\right\rbrace }
\varphi_a^2\\
&\leq
\int_{\left\lbrace |\varphi_a|\geq r\right\rbrace }
\varphi_a^2+
\int_{\left\lbrace s\leq  |\varphi_a|<  r\right\rbrace }
r|\varphi_a|\\
&=\int_{\left\lbrace |\varphi_a|\geq r\right\rbrace }
\varphi_a^2+r\int_{\left\lbrace |\varphi_a|\geq s\right\rbrace }|\varphi_a|-r\int_{\left\lbrace |\varphi_a|\geq r\right\rbrace }|\varphi_a|\\
&\leq
\left( 1-r\right) \int_{\left\lbrace |\varphi_a|\geq r\right\rbrace } \varphi_a^2+
 r \int_{\left\lbrace |\varphi_a|\geq s\right\rbrace }|\varphi_a|\\
 &\leq
 \left( 1-r\right) \int_{\left\lbrace |\varphi_a|\geq r\right\rbrace }|\varphi_a|+
  r \int_{\left\lbrace |\varphi_a|\geq s\right\rbrace }|\varphi_a|.
\end{aligned}
\end{equation}
Thus, for all $0\leq s, r,u\leq 1$ and $s<r$,  by (\ref{equation leqr to geqr varphia}) and (\ref{equation leqr to  varphia2 1-r}) we have
\begin{eqnarray}\label{equation leqr to  estimate varphia2 r}
&~&\int_{\left\lbrace |\varphi_a|\geq s\right\rbrace }\varphi_a^2\\
&\leq&
   r \int_{\left\lbrace |\varphi_a|\geq s\right\rbrace }|\varphi_a|+
 \left( 1-r\right) \int_{\left\lbrace |\varphi_a|\geq r\right\rbrace }|\varphi_a|\nonumber\\
&=&   r \int_{\left\lbrace |\varphi_a|\geq s\right\rbrace }|\varphi_a|+
 (1-r)\left[ u\int_{\left\lbrace |\varphi_a|\geq r\right\rbrace }|\varphi_a|+(1-u) \int_{\left\lbrace |\varphi_a|\geq r\right\rbrace }|\varphi_a|\right] \nonumber\\
&\leq&
    r \int_{\left\lbrace |\varphi_a|\geq s\right\rbrace }|\varphi_a|+
 (1-r)\left[
 \frac{2u\int_{\left\lbrace s\leq |\varphi_a|\leq r\right\rbrace }|\varphi_a|}
 {n\ln\left(\frac{1-s^2}{1-r^2} \right) }
 +(1-u) \int_{\left\lbrace |\varphi_a|\geq r\right\rbrace }|\varphi_a|\right]. \nonumber
\end{eqnarray}
Choosing
$$
\frac{2u_0}{n\ln\left(\frac{1-s^2}{1-r^2} \right) }=
1-u_0,
$$
we have
\begin{equation}\label{equation choose s to  estimate varphia2 }
u_0=\frac{n\ln\left(\frac{1-s^2}{1-r^2} \right)}{2+n\ln\left(\frac{1-s^2}{1-r^2} \right)  }.
\end{equation}
Hence, by (\ref{equation leqr to  estimate varphia2 r}) and (\ref{equation choose s to  estimate varphia2 }) we have
\begin{equation*}\label{equation leqr to  estimate varphia2 rln1-r2}
\begin{aligned}
\int_{\left\lbrace |\varphi_a|\geq s\right\rbrace }\varphi_a^2
  &\leq
    r \int_{\left\lbrace |\varphi_a|\geq s\right\rbrace }|\varphi_a|+
 (1-r) (1-u_0)\left(
 \int_{\left\lbrace s\leq  |\varphi_a|\leq r\right\rbrace }|\varphi_a|+  \int_{\left\lbrace  |\varphi_a|\geq r\right\rbrace }|\varphi_a|
\right) \\
&=\left[  r+(1-r)(1-u_0)\right]  \int_{\left\lbrace |\varphi_a|\geq s\right\rbrace}|\varphi_a|\\
&=\frac{2+nr\ln\left(\frac{1-s^2}{1-r^2} \right)}{2+n\ln\left(\frac{1-s^2}{1-r^2} \right)  }
\int_{\left\lbrace |\varphi_a|\geq s\right\rbrace}|\varphi_a|.
\end{aligned}
\end{equation*}
\end{proof}
In particular, setting $s=0$ in Lemma \ref{lem Volume estimation of minimal hypersurface varphi2a and varphia}, we obtain
\begin{cor}\label{cor Volume estimation of minimal hypersurface varphi2a and varphia s=0}
Let $ M^n$ be a closed immersed, non-totally geodesic,  minimal hypersurface    in  $\mathbb{S}^{n+1}$. Then for all  $a\in \mathbb{S}^{n+1}$, the following inequality holds:
$$
\int_{M}\varphi^2_a\leq
\frac{C_0(n)}{4}\int_{M}|\varphi_a|,
$$
where $C_0(n)=4\inf_{0\leq r\leq 1}
\frac{2-nr\ln\left(1-r^2 \right)}{2-n\ln\left(1-r^2 \right)  }$.
\end{cor}
\section{Proof of  Theorem \ref{thm Volume estimation of minimal hypersurface varphi2a and varphia}}
In this section, we will  prove Theorem \ref{thm Volume estimation of minimal hypersurface varphi2a and varphia} by Lemma \ref{lem Volume estimation of minimal hypersurface}  and Lemma  \ref{lem Volume estimation of minimal hypersurface varphi2a and varphia}.

\begin{proof}[\textbf{Proof of Theorem $\mathbf{\ref{thm Volume estimation of minimal hypersurface varphi2a and varphia}}$}]
\textbf{Case $({\rm i})$.}
Since  $ M^n$ is a closed  minimal hypersurface $($non-totally geodesic$)$   in  $\mathbb{S}^{n+1}$,  by Lemma \ref{lem Volume estimation of minimal hypersurface} we have
\begin{equation}\label{equation varphi2 C1}
\inf_{a\in \mathbb{S}^{n+1}}\int_{M} \varphi^2_a
\geq C_1 {\rm Vol }\left( M^n \right),
\end{equation}
where $C_1=\max\{\theta_1,\theta_2\}$ and
$$
\theta_1=\frac{\int_{M}S}
{2n S_{\max}{\rm Vol }\left( M^n \right)},\ \
\theta_2=\frac{n}{4n^2-3n+1}
 \frac{
 \left( {\int_{M} }S\right) ^2}
{ {\rm Vol }\left( M^n \right) \int_{M} S^2}.
$$
On  one hand,
if $C_1\geq s^2$, then (\ref{equation varphi2 C1}) shows
\begin{equation}\label{equation estimate  varphia2 C1r2}
\begin{aligned}
\int_{\left\lbrace |\varphi_a|\geq s\right\rbrace }
\varphi_a^2
&=
\int_{M}\varphi_a^2-\int_{\left\lbrace |\varphi_a|< s\right\rbrace }
\varphi_a^2\\
&\geq\int_{M}C_1-\int_{\left\lbrace |\varphi_a|< s\right\rbrace }
s^2\\
&=\int_{\left\lbrace |\varphi_a|\geq s\right\rbrace }C_1+\int_{\left\lbrace |\varphi_a|< s\right\rbrace }
\left( C_1-s^2\right) \\
&\geq\int_{\left\lbrace |\varphi_a|\geq s\right\rbrace }C_1.
\end{aligned}
\end{equation}
By Lemma \ref{lem Volume estimation of minimal hypersurface varphi2a and varphia},   (\ref{equation Stokess formula to varphia}) and (\ref{equation estimate  varphia2 C1r2}) , we obtain
$$
\int_{\left\lbrace |\varphi_a|\geq s\right\rbrace }C_1
\leq
\int_{\left\lbrace |\varphi_a|\geq s\right\rbrace }
\varphi_a^2
\leq
C_2
\int_{\left\lbrace |\varphi_a|\geq s\right\rbrace }|\varphi_a|
\leq
C_2
\int_{\left\lbrace |\varphi_a|= s\right\rbrace}\frac{\sqrt{1-s^2}}{n},
$$
where $C_2=\inf_{s\leq r\leq 1}\frac{2+nr\ln\left(\frac{1-s^2}{1-r^2} \right)}{2+n\ln\left(\frac{1-s^2}{1-r^2} \right)  }$.
Thus
\begin{equation}\label{equation estimate  varphia2 C1geq s2}
 {\rm Vol }{\left\lbrace |\varphi_a|=s\right\rbrace }\geq
\frac{nC_1}{C_2\sqrt{1-s^2}}
{\rm Vol }{\left\lbrace |\varphi_a|\geq s\right\rbrace }\ \
\left( \sqrt{C_1}\geq s> 0\right).
\end{equation}
In particular, if $s=0$, then
$$
\begin{aligned}
 \lim_{s\to 0^+} {\rm Vol }{\left\lbrace |\varphi_a|=s\right\rbrace }
 &=
  \lim_{s\to 0^+}  {\rm Vol }{\left\lbrace \varphi_a=s\right\rbrace }
  +  \lim_{s\to 0^+} {\rm Vol }{\left\lbrace \varphi_a=-s\right\rbrace }\\
& 
=2{\rm Vol }{\left\lbrace \varphi_a=0\right\rbrace },
 \end{aligned}
$$
and
$$
 \lim_{s\to 0^+} {\rm Vol }{\left\lbrace |\varphi_a|\geq s\right\rbrace }=
 {\rm Vol }{\left\lbrace |\varphi_a|\geq 0\right\rbrace }=
 {\rm Vol }
\left(  M^n\right).
$$
By (\ref{equation estimate  varphia2 C1geq s2}), one has
\begin{equation}\label{equation estimate  varphia2 s=0}
 {\rm Vol }{\left\lbrace \varphi_a=0\right\rbrace }\geq
\frac{nC_1}{2C_2}
{\rm Vol }{\left\lbrace |\varphi_a|\geq 0\right\rbrace }=
\frac{nC_1}{2C_2}
{\rm Vol }
\left(  M^n\right).
\end{equation}
On the other hand,
by (\ref{equation Stokess formula to varphia}), we have
$$
\int_{\left\lbrace |\varphi_a|\geq s\right\rbrace }s
\leq
\int_{\left\lbrace |\varphi_a|\geq s\right\rbrace }
|\varphi_a|
\leq
\int_{\left\lbrace |\varphi_a|= s\right\rbrace}\frac{\sqrt{1-s^2}}{n}\ \
 \left( 1>s>0 \right).
$$
Hence
\begin{equation}\label{equation estimate  varphia2 C1leq s2}
 {\rm Vol }{\left\lbrace |\varphi_a|=s\right\rbrace }\geq
\frac{ns}{\sqrt{1-s^2}}
{\rm Vol }{\left\lbrace |\varphi_a|\geq s\right\rbrace }\ \
\left( 1> s>0\right).
\end{equation}
Choose
 $$
\frac{ns}{\sqrt{1-s^2}}=\frac{nC_1}{C_2\sqrt{1-s^2}},
$$
which implies that $s=\frac{C_1}{C_2}$. Then
we have the following discussions:
\begin{itemize}
\item [(1)] If $s=0$, (\ref{equation estimate  varphia2 s=0})  implies
 $${\rm Vol }{\left\lbrace \varphi_a=0\right\rbrace }\geq
\frac{nC_1}{2C_2}
{\rm Vol }{\left\lbrace |\varphi_a|\geq 0\right\rbrace }=
\frac{nC_1}{2C_2}
{\rm Vol }
\left(  M^n\right);$$
\item [(2)] If $0<s\leq \min\left\lbrace \sqrt{C_1},\frac{C_1}{C_2}\right\rbrace $, (\ref{equation estimate  varphia2 C1geq s2}) implies
$$
 {\rm Vol }{\left\lbrace |\varphi_a|=s\right\rbrace }\geq
\frac{nC_1}{C_2\sqrt{1-s^2}}
{\rm Vol }{\left\lbrace |\varphi_a|\geq s\right\rbrace };
$$
\item [(3)]  If $\min\left\lbrace \sqrt{C_1},\frac{C_1}{C_2}\right\rbrace<s< 1$, (\ref{equation estimate  varphia2 C1leq s2})
implies
$$
{\rm Vol }{\left\lbrace |\varphi_a|=s\right\rbrace }\geq
\frac{ns}{\sqrt{1-s^2}}
{\rm Vol }{\left\lbrace |\varphi_a|\geq s\right\rbrace }.
$$
\end{itemize}
\textbf{Case $({\rm ii})$.}
By Proposition \ref{prop funda}, we have
$$\nabla \varphi_a=a^{\rm T},\ \   \Delta \varphi_a=-n\varphi_a,$$
for all $a\in \mathbb{S}^{n+1}$.
Hence, by the divergence theorem and $S\not\equiv 0$, 
one has
\begin{equation*}\label{equation Stokess formula to varphia s=0}
\int_{M }
|\varphi_a|
=\int_{\left\lbrace \varphi_a>0\right\rbrace}
\varphi_a-\int_{\left\lbrace \varphi_a\leq 0\right\rbrace}
\varphi_a=
\int_{\left\lbrace |\varphi_a|= 0\right\rbrace}\frac{2|a^{\rm T}|}{n}. 
\end{equation*}
Since
$$
\int_{a\in\mathbb{S}^{n+1}}|\varphi_a|
=2{\rm Vol}\left(\mathbb{B}^{n+1} \right)=
\frac{2}{n+1}{\rm Vol}\left(\mathbb{S}^{n} \right),
$$
we have
$$
\frac{2}{n+1}{\rm Vol}\left(\mathbb{S}^{n} \right){\rm Vol}(M^n)=\int_{a\in\mathbb{S}^{n+1}}
\int_{x\in M}
|\varphi_a|=\int_{a\in\mathbb{S}^{n+1}}\int_{\left\lbrace |\varphi_a|= 0\right\rbrace}\frac{2|a^{\rm T}|}{n}.
$$
By (\ref{equation a2 psi2 vaephi2}), one has
$$
{\rm Vol}(M^n)\leq\frac{\left( n+1\right) {\rm Vol}\left( \mathbb{S}^{n+1}\right) }{n{\rm Vol}\left(\mathbb{S}^{n} \right)}
\sup_{a\in \mathbb{S}^{n+1}}
 {\rm Vol }{\left\lbrace \varphi_a=0\right\rbrace }.
$$
\end{proof}

Combining the intrinsic and extrinsic geometry, Ge and Li generalized Einstein manifolds to  Integral-Einstein (IE)  submanifolds in \cite {Ge Li 2020}.
\begin{defn}\cite{Ge Li 2020}\label{defintion integral Eins manifold}
Let $M^n$ $(n\geq3)$ be a compact submanifold in the Euclidean space $\mathbb{R}^{N}$. Then $M^n$ is an IE submanifold if and only if  for any unit vector $a\in \mathbb{S}^{N-1}$
\begin{equation*}\label{equation integral Einstein manifold}
\int_{M}\left( {\rm Ric}-\frac{R}{n}\mathbf{g}\right)
(a^{\rm T},a^{\rm T})=0,
\end{equation*}
 where $a^{\rm T}\in\Gamma(TM)$ denotes the tangent component of the constant vector $a$ along $M^n$; $ {\rm Ric}$ is the Ricci curvature tensor and $R$ is the scalar curvature.
\end{defn}

\begin{cor}\label{cor Volume estimation of Einstein minimal hypersurface varphi2a and varphia S equiv constant}
Let $ M^n$ be a closed immersed, non-totally geodesic,  minimal hypersurface in $\mathbb{S}^{n+1}$. If it is IE and CSC  (or CSC  with $S>n$ and constant third mean curvature), then for all $0 \leq s<1$ and $a\in \mathbb{S}^{n+1}$, the following inequality holds:
$$
 {\rm Vol }{\left\lbrace |\varphi_a|=s\right\rbrace }\geq
 C(n,s)
{\rm Vol }{\left\lbrace |\varphi_a|\geq s\right\rbrace },
$$
where
$$
 C(n,s)=
 \left\{
\begin{array}{rcl}
    \frac{n}{2(n+2)C_2},& \ \  &s=0;\\
    \frac{n}{(n+2)C_2\sqrt{1-s^2}},&\ \   &0<s\leq \min\left\lbrace \sqrt{\frac{1}{n+2}},\frac{1}{(n+2)C_2}\right\rbrace;\\
    \frac{ns}{\sqrt{1-s^2}},&\ \   &\min\left\lbrace \sqrt{\frac{1}{n+2}},\frac{1}{(n+2)C_2}\right\rbrace< s<1.
\end{array}
 \right.
$$
\end{cor}
\begin{proof}
 If $M^n$ is minimal, IE and CSC, then \cite{Ge Li 2020} showed that
\begin{equation*}\label{equation integral Einstein  hypersurface H=0 S=constant}
\int_{M}\varphi_a^2=\frac{1}{n+2}{\rm Vol }(M^n), \ \ a \in \mathbb{S}^{n+1}.
\end{equation*}
Thus, $C_1=\frac{1}{n+2}$ in Theorem \ref{thm Volume estimation of minimal hypersurface varphi2a and varphia}.
For a closed minimal CSC hypersurface in $\mathbb{S}^{n+1}$  with $S>n$ and constant third mean curvature, Ge-Li proved that
it is an IE hypersurface in \cite{Ge Li 2020}. Thus, Corollary  \ref{cor Volume estimation of Einstein minimal hypersurface varphi2a and varphia S equiv constant} is also true in this case.
\end{proof}

\section{Proof of Theorem \ref{Cheeger isoperimetric constant}}
In this section, we will discuss the Cheeger isoperimetric constant of minimal hypersurfaces   in  $\mathbb{S}^{n+1}$.
\begin{defn}\cite{Cheeger 1970}\label{Definition Cheeger isoperimetric constant}
The Cheeger isoperimetric constant of a closed Riemannian
manifold $M^n$    is defined as:
$$
h(M)=\inf_H\frac{{\rm Vol}{\left( H\right) }}
{\min\left\lbrace {\rm Vol}(M_1),{\rm Vol}(M_2)\right\rbrace},
$$
where the infimum  is taken over all the submanifolds $H$ of codimension $1$ of $M^n$; $M_1$ and $M_2$ are submanifolds of $M^n$ with their boundaries in $H$ and satisfy $M = M_1 \sqcup M_2 \sqcup H$ $($a disjoint union$)$.
\end{defn}

\begin{rem}\label{remark cheeger constant and varpsi}
Let $M^n$ be a closed, immersed, minimal  hypersurface   in  $\mathbb{S}^{n+1}$, which is non-totally geodesic. Since there is a vector
${a\in \mathbb{S}^{n+1}}$ such that
${\rm Vol }{\left\lbrace \varphi_a>0\right\rbrace }= {\rm Vol }{\left\lbrace \varphi_a<0\right\rbrace }$,
we have
$$
h(M)\leq
\sup_{a\in \mathbb{S}^{n+1}}
\frac{2
 {\rm Vol }{\left\lbrace \varphi_a=0\right\rbrace }}
 {{\rm Vol }
 \left(  M^n\right)}.
$$
Moreover, if the image of $M^n$ is invariant under the antipodal map, then ${\rm Vol }{\left\lbrace \varphi_a>0\right\rbrace }= {\rm Vol }{\left\lbrace \varphi_a<0\right\rbrace }$ for all $a\in \mathbb{S}^{n+1}$ and 
$$
h(M)\leq
\inf_{a\in \mathbb{S}^{n+1}}
\frac{2
 {\rm Vol }{\left\lbrace \varphi_a=0\right\rbrace }}
 {{\rm Vol }
 \left(  M^n\right)}.
$$
\end{rem}

In 1970,  Cheeger \cite{Cheeger 1970} gave the famous inequality between the first positive eigenvalue $\lambda_1(M)$ of the Laplacian and the Cheeger isoperimetric constant $h(M)$ (see Definition \ref{Definition Cheeger isoperimetric constant}):
$$
h^2(M)\leq
4\lambda_1(M).
$$
Obviously, $\lambda_1(M)\leq n$ for minimal hypersurfaces in $\mathbb{S}^{n+1}$ because of Proposition  \ref{prop funda} and we have
$$
h(M)\leq
2\sqrt{\lambda_1(M)}\leq 2\sqrt{n}.
$$
The Yau Conjecture \cite{Schoen Richard and S. T. Yau 1994} asserts that \emph{
if $ M^n$ is a closed   embedded minimal hypersurface of $\mathbb{S}^{n+1}$, then
 $\lambda_1 (M)=n$.}
 In particular,
 Choi and Wang \cite{Choi Wang 1983} showed that $\lambda_1(M)\geq n/2$ and a careful argument (see \cite[ Theorem 5.1]{Brendle S 2013 survey of recent results}) implied that the strict inequality holds, i.e., $\lambda_1 (M)> n/2$. In addition,
 Tang and Yan \cite{TXY14,TY13} proved the Yau Conjecture in the isoparametric case.
 Choe and Soret \cite{J Choe and M Soret  2009} were  able to verify the Yau Conjecture for the Lawson surfaces and the Karcher-Pinkall-Sterling examples.
For more details and references,  please see the elegant survey by Brendle \cite{Brendle S 2013 survey of recent results}. Besides,  Buser \cite{Buser 1982 isoperimetric} proved that

\begin{lem}[\bf Buser \cite{Buser 1982 isoperimetric}]\label{Buser 1982 isoperimetric}
 If the Ricci curvature of a closed Riemannian manifold ${M}^{n}$ is bounded below by $-(n-1) \delta^{2}\ (\delta \geq 0)$, then
\begin{equation}\label{equation Buser}
\lambda_{1} (M)\leq 2 \delta(n-1) h(M)+10 h^{2}(M).
\end{equation}
\end{lem}

Next, we will prove Theorem \ref{Cheeger isoperimetric constant}  by Lemma \ref{lem Volume estimation of minimal hypersurface}, Lemma \ref{Buser 1982 isoperimetric} and Corollary \ref{cor Volume estimation of minimal hypersurface varphi2a and varphia s=0}.

\begin{proof}[\textbf{Proof of Theorem $\mathbf{\ref{Cheeger isoperimetric constant}}$}]
Without loss of generality, assuming that
$ {\rm Vol }{\left\lbrace \varphi_a>0\right\rbrace}\geq 
{\rm Vol }{\left\lbrace \varphi_a<0\right\rbrace }$, one has
\begin{equation}\label{equation  cheeger constant}
h(M)\leq
\frac{
 {\rm Vol }{\left\lbrace \varphi_a=0\right\rbrace }}
 {{\rm Vol }{\left\lbrace \varphi_a<0\right\rbrace }}.
\end{equation}
For $ {\rm Vol }{\left\lbrace \varphi_a>0\right\rbrace}\leq 
{\rm Vol }{\left\lbrace \varphi_a<0\right\rbrace }$, the proof is similar and the following estimates of inequalities can be found in Ge and Li \cite{Ge Li 2022 Volume gap}.
By Proposition \ref{prop funda}, for any $a\in {\mathbb{S}^{n+1}}$, $\int_{M}\varphi_{a}=0.$ Thus
\begin{equation}\label{equation  varphia leq 0}
\int_{{\left\lbrace \varphi_a> 0\right\rbrace  }}\varphi_{a}=\int_{\left\lbrace \varphi_a
< 0\right\rbrace  }
-\varphi_a=\frac{1}{2}\int_{M}|\varphi_a|.
\end{equation}
The divergence theorem  shows that
$$
\int_{\left\lbrace \varphi_a< 0\right\rbrace }\Delta \varphi_a^2=0,
$$
and by
$\Delta \varphi_a^2=-2n\varphi_a^2+2|a^{\rm T}|^2$, one has
\begin{equation}\label{equation  n varphia2 leq at2}
n\int_
{\left\lbrace \varphi_a< 0\right\rbrace }
\varphi_a^2=
\int_{\left\lbrace \varphi_a< 0\right\rbrace }
|a^{\rm T}|^2.
 \end{equation}
Then, due to  (\ref{equation a2 psi2 vaephi2}) and (\ref{equation  n varphia2 leq at2}),  we have
\begin{equation}\label{equation  varphia2 leq 0}
(n+1)\int_{{\left\lbrace \varphi_a< 0\right\rbrace }}\varphi_a^2\leq
\int_{\left\lbrace \varphi_a< 0\right\rbrace  }
 1.
 \end{equation}
 By the Cauchy-Schwarz inequality and (\ref{equation  varphia2 leq 0}), one has
 \begin{equation}\label{equation  varphia1 2 leq 0}
 \sqrt{\frac{1}{n+1}
 }
  \int_{\left\lbrace \varphi_a< 0\right\rbrace  }1
  \geq
\sqrt{
 \int_{\left\lbrace \varphi_a< 0\right\rbrace  }1
 \int_{{\left\lbrace \varphi_a< 0\right\rbrace }}\varphi_a^2
}
 \geq
 \int_{\left\lbrace \varphi_a< 0\right\rbrace }
-\varphi_a.
  \end{equation}
By Corollary \ref{cor Volume estimation of minimal hypersurface varphi2a and varphia s=0}, (\ref{equation  cheeger constant}),  (\ref{equation  varphia leq 0}) and (\ref{equation  varphia1 2 leq 0}), we have
$$
\frac{ {\rm Vol }{\left\lbrace \varphi_a=0\right\rbrace }}
{h(M)}
\geq
 {\rm Vol }{\left\lbrace \varphi_a<0\right\rbrace }\geq
\frac{\sqrt{n+1}}{2}\int_{M}|\varphi_a|\geq
\frac{2\sqrt{n+1}}{C_0(n)}\int_{M}\varphi^2_a.
$$
Hence,   by Lemma \ref{lem Volume estimation of minimal hypersurface} we have
\begin{equation*}
 {\rm Vol }{\left\lbrace \varphi_a=0\right\rbrace }\geq
 \frac{2\sqrt{n+1}}{C_0(n)}h(M)\int_{M}\varphi^2_a\geq
 \frac{2\sqrt{n+1}C_1}{C_0(n)}h(M) {\rm Vol }{\left( M^n\right) }.
\end{equation*}
\textbf{Case $({\rm i})$.}
Since $M^n$ is a minimal hypersurface in  $\mathbb{S}^{n+1}$, the Ricci curvature 
 is given by
$$
{\rm Ric}\left( X,Y\right)=(n-1)g\left( X,Y\right)-g\left(AX,AY \right),\ \ \ \ \      X,Y\in \mathfrak{X}(M).
$$
Let  $\lambda_1(A), \lambda_2(A),\cdots,\lambda_n(A) $ denote the  eigenvalues of the shape operator $A$. 
We obtain
$$\sum_{i=1}^{n}\lambda_i=0,\ \sum_{i=1}^{n}\lambda_i^2=\|A\|^2=S,$$
and
$$
\begin{aligned}
0=\sum_{i,j=1}^{n}\lambda_i\lambda_j
&=\lambda_1^2+
2\sum_{j=2}^{n}\lambda_1\lambda_j+
\sum_{i,j=2}^{n}\lambda_i\lambda_j\\
&\leq-\lambda_1^2+\sum_{i,j=2}^{n}\frac{\lambda_i^2+\lambda_j^2}{2}\\
&=(n-1)S-n\lambda_1^2.
\end{aligned}
$$
Thus
$$
{\rm Ric}\left( X,X\right)\geq\left(  n-1-\lambda_1^2\right) g\left(X,X \right)
\geq -(n-1)\frac{S-n}{n}g\left(X,X \right).
$$
By Lemma \ref{Buser 1982 isoperimetric} and 
$\lambda_1 (M)> n/2$ (cf. Choi-Wang \cite{Choi Wang 1983} and Brendle \cite{Brendle S 2013 survey of recent results}), one has
\begin{equation*}
\frac{n}{2}< \lambda_1(M)\leq 2 \delta(n-1) h(M)+10 h^{2}(M).
\end{equation*}
Note that $S_{\max}\geq n$ for all non-totally geodesic minimal hypersurfaces in $\mathbb{S}^{n+1}$ by Simons' inequality \cite{Simons68}
$$
\int_{M}S\left( S-n\right)\geq0.
$$
Setting $\delta=\sqrt{\frac{S_{\max}-n}{n}}$, we have
$$
h(M)>\frac{-\delta(n-1)+\sqrt{\delta^2(n-1)^2+5n}}{10}.
$$
\textbf{Case $({\rm ii})$.}
If the image of $M^n$ is invariant under the antipodal map, the proof is complete by Remark \ref{remark cheeger constant and varpsi}.  
\end{proof}
\begin{rem}
If ${M}^{n}$ is a minimal isoparametric hypersurface with $g\geq 2 $ distinct principal curvatures  in  $\mathbb{S}^{n+1}$, then  $\lambda_1 (M)=n$
$($cf. Tang-Yan \cite{TY13}$)$,
$S\equiv(g-1)n$ and $\delta=\sqrt{g-2}$
$\left( 2 \leq g\leq 6\right) $.
Thus, (\ref{equation Buser}) implies that
$$
h(M)\geq \frac{-\sqrt{g-2}(n-1)+\sqrt{(g-2)(n-1)^2+10n}}{10}.$$
In fact, Muto \cite{Muto 1988} carefully estimated the Cheeger isoperimetric constant of  minimal isoparametric hypersurfaces and got better results.
\end{rem}
\begin{rem}
Let  $M^n$ be a closed embedded minimal  hypersurface in  $\mathbb{S}^{n+1}$. If  $S<c(n)$ and $c(n)$ depends only on $n$, then there is  a positive constant
   $\eta(n)>0$, depending only on $n$,
   such that
$ h(M)>\eta(n).$
\end{rem}


\begin{acknow}
The authors sincerely thank  the referee for their  many valuable and constructive suggestions, which made this paper more readable.  They also thank Professors Zizhou Tang and  Jianquan Ge for their very detailed and valuable comments.
\end{acknow}


\begin{thebibliography}{99}
\bibitem{Brendle S 2013 survey of recent results}
S. Brendle,   \emph{Minimal surfaces in $S^3$: a survey of recent results}, Bull. Math. Sci. \textbf{3} (2013), 133--171.

\bibitem{Brendle S The isoperimetric inequality JMS}
S. Brendle, \emph{The isoperimetric inequality for a minimal submanifold in Euclidean space}, J. Amer. Math. Soc. \textbf{34} (2021),  595--603.

\bibitem{Buser 1982 isoperimetric}
P. Buser, \emph{
A note on the isoperimetric constant,}
Ann. Sci. {$\rm\acute{E}$}cole Norm. Sup. (4) \textbf{15} (1982), 213--230.

  \bibitem{Chang 1993}
 S. P. Chang, \emph{On minimal hypersurfaces with constant scalar curvatures in $\mathbb{S}^4$}, J. Diff. Geom.  \textbf{37} (1993), 523--534.

\bibitem{Cheeger 1970}
  J. Cheeger, $\emph{A Lower Bound for the Smallest Eigenvalue of the Laplacian}$, Problems in analysis, (Papers dedicated to Salomon Bochner, 1969), pp 195--199, N. J.  Princeton Univ. Press, Princeton, N. J., 1970.

\bibitem{Cheng Li Yau 1984 Heat equations}
  S. Y. Cheng, P. Li and S. T. Yau, \emph{Heat equations on minimal submanifolds and their applications}, Amer. J. Math. \textbf{106} (1984), 1033--1065.


\bibitem{J Choe and Gulliver 1992}
J. Choe and R.  Gulliver, \emph{ Isoperimetric inequalities on minimal submanifolds of space forms.}  Manuscripta Math. \textbf{77} (1992),  169--189.




\bibitem{J Choe and M Soret  2009}
J. Choe and  M. Soret, \emph{First eigenvalue of symmetric minimal surfaces in $\mathbb{S}^3$}, Indiana Univ. Math. J. \textbf{58}  (2009) 269--281

\bibitem{Choi Wang 1983}
H. I. Choi and  A. N. Wang,  \emph{A first eigenvalue estimate for minimal hypersurfaces}, J. Diff. Geom. \textbf{18} (1983), 559--562.

\bibitem{Ge Li 2020}
J. Q. Ge and F. G. Li, \emph{Integral-Einstein hypersurfaces in spheres}, arXiv:2101.03753.

\bibitem{Ge Li 2022 Volume gap}
J. Q. Ge and F. G. Li, \emph{Volume gap for minimal submanifolds in spheres},
arXiv:2210.04654.



 \bibitem{Muto 1988}
H. Muto, \emph{The first eigenvalue of the Laplacian of an isoparametric minimal hypersurface in a unit sphere,} Math. Z. \textbf{197} (1988), 531--549,

 \bibitem{Nomizu and Smyth 1969}
 K. Nomizu and B. Smyth, \emph{On the Gauss Mapping for Hypersurfaces of Constant Mean Curvature in the Sphere}, Comm. Math. Helv. \textbf{44} (1969), 484--490.

    \bibitem{Osserman Robert 1980 isoperimetric}
 R. Osserman, \emph{The isoperimetric inequality}, Bull. Amer. Math. Soc. \textbf{84} (1978).




  \bibitem{M Scherfner S Weiss and  Yau 2012}
  M. Scherfner, S. Weiss and S. T. Yau, \emph{A review of the Chern conjecture for isoparametric hypersurfaces in spheres}, In: Advances in Geometric Analysis, pp. 175--187, Adv. Lect. Math. (ALM), \textbf{21}, Int. Press, Somerville, MA, 2012.

  \bibitem{Schoen Richard and S. T. Yau 1994}
  R. Schoen, and S. T. Yau, \emph{Lectures on Differential Geometry},  International Press, 1994.

 \bibitem{Simons68}
    J. Simons, \emph{Minimal varieties in Riemannian manifolds}, Ann. Math. \textbf{88} (1968), 62--105.

 \bibitem{TWY18}
Z. Z. Tang, D. Y. Wei and W. J. Yan, \emph{A sufficient condition for a hypersurface to be isoparametric,}  Tohoku Math. J.  \textbf{72} (2020), 493--505.


 \bibitem{TXY14}
Z. Z. Tang, Y. Q. Xie and W. J. Yan, \emph{ Isoparametric foliation and Yau conjecture on the first eigenvalue, II.}  J. Funct. Anal. \textbf{266} (2014), 6174--6199.


 \bibitem{TY13}
Z. Z. Tang and W. J. Yan, \emph{Isoparametric foliation and Yau conjecture on the first eigenvalue}, J. Diff. Geom.  \textbf{94} (2013), 521--540.

 \bibitem{TY20}  
 Z. Z. Tang and W. J. Yan, \emph{On the Chern conjecture for isoparametric hypersurfaces},  Sci. China Math. \textbf{66} (2023),  143--162. 



 \bibitem{Takahashi 1966}
 T. Takahashi, \emph{Minimal immersion of Riemannian manifolds}, J. Math. Soc. Japan \textbf{18} (1966), 380--385.

 \bibitem{Cheng Qing Ming and H Yang 1998}
    H. C. Yang and Q. M. Cheng, \emph{Chern's conjecture on minimal hypersurfaces}. Math. Z. \textbf{227} (1998), 377--390.

 \bibitem{Yau Isoperimetric constants 1975}
 S. T. Yau, \emph{Isoperimetric constants and the first eigenvalue of a compact manifold.} Ann. Sci. Ec. Norm. Super., Paris.   \textbf{8} (1975), 487--507.

\end{thebibliography}
\end{document}